\documentclass{article}
\usepackage{graphicx}
\usepackage{amsmath}
\usepackage{amsfonts}
\usepackage{amssymb}
\newtheorem{theorem}{Theorem}

\newtheorem{corollary}[theorem]{Corollary}

\newtheorem{definition}[theorem]{Definition}

\newtheorem{lemma}[theorem]{Lemma}
\newtheorem{notation}[theorem]{Notation}

\newtheorem{proposition}[theorem]{Proposition}
\newtheorem{remark}[theorem]{Remark}

\newenvironment{proof}[1][Proof]{\textbf{#1.} }{\ \rule{0.5em}{0.5em}}

\begin{document}

\title{Axiomatic Differential Geometry III-3\\-Its Landscape-\\Chapter 3: The Old Kingdom of Differential Geometers}
\author{Hirokazu NISHIMURA\\Institute of Mathematics\\University of Tsukuba\\Tsukuba, Ibaraki, 305-8571, JAPAN}
\maketitle
\begin{abstract}
The principal objective of this paer is to study the relationship between the
old kingdom of differential geometry (the category of smooth manifolds) and
its new kingdom (the category of functors on the category of Weil algebras to
some smooth category). It is shown that the canonical embedding of the old
kingdom into the new kingdom preserves Weil functors.
\end{abstract}

\section{Introduction}

Roughly speaking, the path to axiomatic differential geometry is composed of
five acts. \textbf{Act One} was Weil's algebraic treatment of nilpotent
infinitesimals in \cite{wei}, namely, the introduction of so-called
\textit{Weil algebras}. It showed that nilpotent infinitesimals could be
grasped algebraically. While nilpotent infinitesimals are imaginary entities,
Weil algebras are real ones. \textbf{Act Two} began almost at the same time
with Steenrod's introduction of \textit{convenient categories} of topological
spaces (cf. \cite{st}), consisting of a string of proposals of convenient
categories of smooth spaces. Its principal slogan was that the category of
differential geometry should be (locally) cartesian closed. The string was
panoramized by \cite{stacey}\ as well as \cite{baez}. \textbf{Act
Three}$\mathbb{\ }$was so-called \textit{synthetic differential geometry}, in
which synthetic methods as well as nilpotent infinitesimals play a predominant
role. It demonstrated amply that differential geometry could be made
\textit{axiomatic }in the same sense that Euclidean geometry is so, though it
should resort to reincarnation of nilpotent infinitesimals. In any case,
synthetic differential geometers were forced to fabricate their own world,
called \textit{well-adapted models}, where they could indulge in their
favorite nilpotent infinitesimals incessantly. Their unblushing use of
moribund nilpotent infinitesimals alienated most of orthodox mathematicians,
because nilpotent infinitesimals were almost eradicated as genuine hassle and
replaced by so-called $\varepsilon-\delta$ arguments in the 19th century. The
reader is referred to \cite{kock} and \cite{lav} for good treatises on
synthetic differential geometry. \textbf{Act Four} was the introduction of
\textit{Weil functors} and their thorough study by what was called the Czech
school of differential geometers in the 1980's, for which the reader is
referred to Chapter VIII of \cite{kol} and \S 31 of \cite{kri}. Weil functors,
which are a direct generalization of the tangent bundle functor, opens a truly
realistic path of axiomatizing differential geometry without nilpotent
infinitesimals. Then \textbf{Act Five} is our \textit{axiomatic differential
geometry}, which is tremendously indebted to all previous four acts. For
axiomatic differential geometry, the reader is referred to \cite{nishi3},
\cite{nishi4}, \cite{nishi5}, \cite{nishi6}, \cite{nishi7}, \cite{nishi9} and
\cite{nishi10}.

In our previous two papers \cite{nishi9} and \cite{nishi10}, we have developed
model theory for axiomatic differential geometry, in which the category
$\mathcal{K}_{\mathbf{Smooth}}$\ of functors on the category $\mathbf{Weil}%
_{\mathbf{R}}$\ of Weil algebras to the smooth category $\mathbf{Smooth}$\ (by
which we mean any proposed or possible convenient category of smooth spaces)
and their natural transformations play a crucial role. We will study the
relationship between the category $\mathbf{Mf}$ of smooth manifolds and smooth
mappings and our new kingdom $\mathcal{K}_{\mathbf{Smooth}}$\ as well as that
between $\mathbf{Smooth}$\ and $\mathcal{K}_{\mathbf{Smooth}}$ in this paper.

\section{Convenient Categories of Smooth Spaces}

The category of topological spaces and continuous mappins is by no means
cartesian closed. In 1967 Steenrod \cite{st} popularized the idea of
\textit{convenient category} by announcing that the category of compactly
generated spaces and continuous mappings renders a good setting for algebraic
topology. The proposed category is cartesian closed, complete and cocomplete,
and contains all CW complexes.

At about the same time, an attempt to give a convenient category of smooth
spaces began, and we have a few candidates at present. For a thorough study
upon the relationship among these already proposed candidates, the reader is
referred to \cite{stacey}, in which he or she will find, by way of example,
that the category of Fr\"{o}licher spaces is a full subcategory of that of
Souriau spaces, and the category of Souriau spaces is in turn a full
subcategory of that of Chen spaces. We have no intention to discuss which is
the best convenient category of smooth spaces here, but we note in passing
that both the category of Souriau spaces and that of Chen spaces are locally
cartesian closed, while that of Fr\"{o}licher spaces is not. At present we
content ourselves with denoting some of such convenient categories of smooth
spaces by $\mathbf{Smooth}$, which is required to be complete and cartesian
closed at least, containing the category $\mathbf{Mf}$ of smooth manifolds as
a full subcategory. Obviously the category $\mathbf{Mf}$\ contains the set
$\mathbf{R}$\ of real numbers.

\section{Weil Functors}

Weil algebras were introduced by Weil himself \cite{wei}. For a thorough
treatment of Weil algebras as smooth algebras, the reader is referred to III.5
in \cite{kock}.

\begin{notation}
We denote by $\mathbf{Weil}_{\mathbf{R}}$\ the category of Weil algebras over
$\mathbf{R}$.
\end{notation}

Let us endow the category $\mathbf{Smooth}$\ with Weil functors.

\begin{proposition}
Let $W$\ be an object in the category $\mathbf{Weil}_{\mathbf{R}}$\ with its
finite presentation
\[
W=C^{\infty}\left(  \mathbf{R}^{n}\right)  /I
\]
as a smooth algebra in the sense of III.5 of \cite{kock}. Let $X,Y\in
\mathbf{Smooth}$, $f,g\in\mathbf{Smooth}\left(  \mathbf{R}^{n},X\right)  $,
and $h\in\mathbf{Smooth}\left(  X,Y\right)  $. If
\[
f\thicksim_{W}g\text{,}%
\]
then
\[
h\circ f\thicksim_{W}h\circ g
\]
\end{proposition}

\begin{proof}
Given $\varsigma\in\mathbf{Smooth}\left(  Y,\mathbf{R}\right)  $, we have
\begin{align*}
&  \varsigma\circ\left(  h\circ f\right)  -\varsigma\circ\left(  h\circ
g\right) \\
&  =\left(  \varsigma\circ h\right)  \circ f-\left(  \varsigma\circ h\right)
\circ g\in I
\end{align*}
so that we have the desired result.
\end{proof}

\begin{corollary}
We can naturally make $\underline{\mathbf{T}}_{\mathbf{Smooth}}^{W}$\ a
functor
\[
\underline{\mathbf{T}}_{\mathbf{Smooth}}^{W}:\mathbf{Smooth}\rightarrow
\mathbf{Smooth}%
\]
\end{corollary}

\begin{proposition}
Let $W_{1}$\ and $W_{2}$\ be objects in the category $\mathbf{Weil}%
_{\mathbf{R}}$\ with their finite presentations
\begin{align*}
W_{1}  &  =C^{\infty}\left(  \mathbf{R}^{n}\right)  /I\\
W_{2}  &  =C^{\infty}\left(  \mathbf{R}^{m}\right)  /J
\end{align*}
as smooth algebras. Let
\[
\varphi:W_{1}\rightarrow W_{2}%
\]
be a morphism in the category $\mathbf{Weil}_{\mathbf{R}}$, so that there
exists a morphism
\[
\overleftarrow{\varphi}:\mathbf{R}^{m}\rightarrow\mathbf{R}^{n}%
\]
in the category $\mathbf{Smooth}$\ such that the composition with
$\overleftarrow{\varphi}$\ renders a mapping
\[
C^{\infty}\left(  \mathbf{R}^{n}\right)  \rightarrow C^{\infty}\left(
\mathbf{R}^{m}\right)
\]
inducing $\varphi$. Let $X\in\mathbf{Smooth}$ and $f,g\in\mathbf{Smooth}%
\left(  \mathbf{R}^{n},X\right)  $. If
\[
f\thicksim_{W_{1}}g
\]
then
\[
f\circ\overleftarrow{\varphi}\thicksim_{\mathbf{W}_{2}}g\circ\overleftarrow
{\varphi}%
\]
\end{proposition}

\begin{proof}
Given any $\varsigma\in\mathbf{Smooth}\left(  Y,\mathbf{R}\right)  $, we have
\begin{align*}
&  \varsigma\circ\left(  f\circ\overleftarrow{\varphi}\right)  -\varsigma
\circ\left(  g\circ\overleftarrow{\varphi}\right) \\
&  =\left(  \varsigma\circ f\right)  \circ\overleftarrow{\varphi}-\left(
\varsigma\circ g\right)  \circ\overleftarrow{\varphi}\\
&  =\left(  \varsigma\circ f-\varsigma\circ g\right)  \circ\overleftarrow
{\varphi}\in J\text{\ }%
\end{align*}
since $\varsigma\circ f-\varsigma\circ g\in I$, and the composition with
$\overleftarrow{\varphi}:\mathbf{R}^{n}\rightarrow\mathbf{R}^{m}$ maps
$I$\ into $J$.
\end{proof}

\begin{corollary}
The above procedure automatically induces a natural transformation
\[
\underline{\alpha}_{\varphi}^{\mathbf{Smooth}}:\underline{\mathbf{T}%
}_{\mathbf{Smooth}}^{W_{1}}\Rightarrow\underline{\mathbf{T}}_{\mathbf{Smooth}%
}^{W_{2}}%
\]
\end{corollary}

\begin{notation}
Given an object $W$\ in the category $\mathbf{Weil}_{\mathbf{R}}$, the
restriction of the functor $\underline{\mathbf{T}}_{\mathbf{Smooth}}^{W}$\ to
the category $\mathbf{Mf}$\ is denoted by $\underline{\mathbf{T}}%
_{\mathbf{Mf}}^{W}$. Given a morphism $\varphi:W_{1}\rightarrow W_{2}$ in the
category $\mathbf{Weil}_{\mathbf{R}}$, the corresponding restriction of
$\underline{\alpha}_{\varphi}^{\mathbf{Smooth}}$ is denoted by $\underline
{\alpha}_{\varphi}^{\mathbf{Mf}}$.
\end{notation}

\begin{remark}
Weil functors
\[
\underline{\mathbf{T}}_{\mathbf{Mf}}^{W}:\mathbf{Mf}\rightarrow\mathbf{Mf}%
\]
are given distinct (but equivalent) definitions and studied thoroughly in
Chapter VIII of \cite{kol} in the finite-dimensional case and \S31 of
\cite{kri} in the infinite-dimensional case.
\end{remark}

It is well known that

\begin{proposition}
We have the following:

\begin{enumerate}
\item  Given an object $W$ in the category $\mathbf{Weil}_{\mathbf{R}}$, the
functor
\[
\underline{\mathbf{T}}_{\mathbf{Mf}}^{W}:\mathbf{Mf}\rightarrow\mathbf{Mf}%
\]
abides by the following conditions:

\begin{itemize}
\item $\underline{\mathbf{T}}_{\mathbf{Mf}}^{W}$ preserves finite products.

\item  The functor
\[
\underline{\mathbf{T}}_{\mathbf{Mf}}^{\mathbf{R}}:\mathbf{Mf}\rightarrow
\mathbf{Mf}%
\]
is the identity functor.

\item  We have
\[
\underline{\mathbf{T}}_{\mathbf{Mf}}^{W_{2}}\circ\underline{\mathbf{T}%
}_{\mathbf{Mf}}^{W_{1}}=\underline{\mathbf{T}}_{\mathbf{Mf}}^{W_{1}%
\otimes_{\mathbf{R}}W_{2}}%
\]
for any objects $W_{1}$ and $W_{2}$ in the category $\mathbf{Weil}%
_{\mathbf{R}}$.
\end{itemize}

\item  Given a morphism $\varphi:W_{1}\rightarrow W_{2}$ in the category
$\mathbf{Weil}_{\mathbf{R}}$, $\underline{\alpha}_{\varphi}^{\mathbf{Mf}%
}:\underline{\mathbf{T}}_{\mathbf{Mf}}^{W_{1}}\Rightarrow\underline
{\mathbf{T}}_{\mathbf{Mf}}^{W_{2}}$ is a natural transformation subject to the
following conditions:

\begin{itemize}
\item  We have
\[
\underline{\alpha}_{\mathrm{id}_{W}}^{\mathbf{Mf}}=\mathrm{id}_{\underline
{\mathbf{T}}_{\mathbf{Mf}}^{W}}%
\]
for any identity morphism $\mathrm{id}_{W}:W\rightarrow W$ in the category
$\mathbf{Weil}_{\mathbf{R}}$.

\item  We have
\[
\underline{\alpha}_{\psi}^{\mathbf{Mf}}\cdot\underline{\alpha}_{\varphi
}^{\mathbf{Mf}}=\underline{\alpha}_{\psi\circ\varphi}^{\mathbf{Mf}}%
\]
for any morphisms $\varphi:W_{1}\rightarrow W_{2}$ and $\psi:W_{2}\rightarrow
W_{3}$ in the category $\mathbf{Weil}_{\mathbf{R}}$.

\item  Given an object $W$\ and a morphism $\varphi:W_{1}\rightarrow W_{2}%
$\ in the category $\mathbf{Weil}_{\mathbf{R}}$, the diagrams
\[%
\begin{array}
[c]{ccc}%
\underline{\mathbf{T}}_{\mathbf{Mf}}^{W}\circ\underline{\mathbf{T}%
}_{\mathbf{Mf}}^{W_{1}} &
\begin{array}
[c]{c}%
\underline{\mathbf{T}}_{\mathbf{Mf}}^{W}\circ\underline{\alpha}_{\varphi
}^{\mathbf{Mf}}\\
\Rightarrow
\end{array}
& \underline{\mathbf{T}}_{\mathbf{Mf}}^{W}\circ\underline{\mathbf{T}%
}_{\mathbf{Mf}}^{W_{2}}\\
\parallel &  & \parallel\\
\underline{\mathbf{T}}_{\mathbf{Mf}}^{W_{1}\otimes_{\mathbf{R}}W} &
\begin{array}
[c]{c}%
\Rightarrow\\
\underline{\alpha}_{\varphi\otimes_{\mathbf{R}}\mathrm{id}_{W}}^{\mathbf{Mf}}%
\end{array}
& \underline{\mathbf{T}}_{\mathbf{Mf}}^{W_{2}\otimes_{\mathbf{R}}W}%
\end{array}
\]
and
\[%
\begin{array}
[c]{ccc}%
\underline{\mathbf{T}}_{\mathbf{Mf}}^{W\otimes_{\mathbf{R}}W_{1}} &
\begin{array}
[c]{c}%
\underline{\alpha}_{\mathrm{id}_{W}\otimes_{\mathbf{R}}\varphi}^{\mathbf{Mf}%
}\\
\Rightarrow
\end{array}
& \underline{\mathbf{T}}_{\mathbf{Mf}}^{W\otimes_{\mathbf{R}}W_{2}}\\
\parallel &  & \parallel\\
\underline{\mathbf{T}}_{\mathbf{Mf}}^{W_{1}}\circ\underline{\mathbf{T}%
}_{\mathbf{Mf}}^{W} &
\begin{array}
[c]{c}%
\Rightarrow\\
\underline{\alpha}_{\varphi}^{\mathbf{Mf}}\circ\underline{\mathbf{T}%
}_{\mathbf{Mf}}^{W}%
\end{array}
& \underline{\mathbf{T}}_{\mathbf{Mf}}^{W_{2}}\circ\underline{\mathbf{T}%
}_{\mathbf{Mf}}^{W}%
\end{array}
\]
are commutative.
\end{itemize}

\item  Given an object $W$ in the category $\mathbf{Weil}_{\mathbf{R}}$, we
have
\[
\underline{\mathbf{T}}^{W}\left(  \mathbf{R}\right)  =W
\]

\item  Given a morphism $\varphi:W_{1}\rightarrow W_{2}$ in the category
$\mathbf{Weil}_{\mathbf{R}}$, we have
\[
\underline{\alpha}_{\varphi}\left(  \mathbf{R}\right)  =\varphi
\]
\end{enumerate}
\end{proposition}

\section{A New Kingdom for Differential Geometers}

\begin{notation}
We introduce the following notation:

\begin{enumerate}
\item  We denote by $\mathcal{K}_{\mathbf{Smooth}}$\ the category whose
objects are functors from the category $\mathbf{Weil}_{\mathbf{R}}$ to the
category $\mathbf{Smooth}$ and whose morphisms are their natural transformations.

\item  Given an object $W$ in the category $\mathbf{Weil}_{\mathbf{R}}$, we
denote by
\[
\mathbf{T}_{\mathbf{Smooth}}^{W}:\mathcal{K}_{\mathbf{Smooth}}\rightarrow
\mathcal{K}_{\mathbf{Smooth}}%
\]
the functor obtained as the composition with the functor
\[
W\otimes_{\mathbf{R}}\cdot:\mathbf{Weil}_{\mathbf{R}}\rightarrow
\mathbf{Weil}_{\mathbf{R}}%
\]
so that for any object $M$\ in the category $\mathcal{K}_{\mathbf{Smooth}}$,
we have
\[
\mathbf{T}_{\mathbf{Smooth}}^{W}\left(  M\right)  =M\left(  W\otimes
_{\mathbf{R}}\cdot\right)
\]

\item  Given a morphism $\varphi:W_{1}\rightarrow W_{2}$ in the category
$\mathbf{Weil}_{\mathbf{R}}$, we denote by
\[
\alpha_{\varphi}^{\mathbf{Smooth}}:\mathbf{T}_{\mathbf{Smooth}}^{W_{1}%
}\Rightarrow\mathbf{T}_{\mathbf{Smooth}}^{W_{2}}%
\]
the natural transformation such that, given an object $W$\ in the category
$\mathbf{Weil}_{\mathbf{R}}$, the morphism
\[
\alpha_{\varphi}^{\mathbf{Smooth}}\left(  M\right)  :\mathbf{T}%
_{\mathbf{Smooth}}^{W_{1}}\left(  M\right)  \rightarrow\mathbf{T}%
_{\mathbf{Smooth}}^{W_{2}}\left(  M\right)
\]
is
\[
M\left(  \varphi\otimes_{\mathbf{R}}\mathrm{id}_{W}\right)  :M\left(
W_{1}\otimes_{\mathbf{R}}W\right)  \rightarrow M\left(  W_{2}\otimes
_{\mathbf{R}}W\right)
\]

\item  We denote by $\mathbb{R}_{\mathbf{Smooth}}$ the functor
\[
\mathbf{R}\mathbb{\otimes}_{\mathbf{R}}\cdot:\mathbf{Weil}_{\mathbf{R}%
}\rightarrow\mathbf{Smooth}%
\]
\end{enumerate}
\end{notation}

We have established the following proposition in \cite{nishi9} and
\cite{nishi10}.

\begin{proposition}
We have the following:

\begin{enumerate}
\item $\mathcal{K}_{\mathbf{Smooth}}$ is a category which is complete and
cartesian closed.

\item  Given an object $W$ in the category $\mathbf{Weil}_{\mathbf{R}}$, the
functor
\[
\mathbf{T}_{\mathbf{Smooth}}^{W}:\mathcal{K}_{\mathbf{Smooth}}\rightarrow
\mathcal{K}_{\mathbf{Smooth}}%
\]
abides by the following conditions:

\begin{itemize}
\item $\mathbf{T}_{\mathbf{Smooth}}^{W}$ preserves limits.

\item  The functor
\[
\mathbf{T}_{\mathbf{Smooth}}^{\mathbf{R}}:\mathcal{K}_{\mathbf{Smooth}%
}\rightarrow\mathcal{K}_{\mathbf{Smooth}}%
\]
is the identity functor.

\item  We have
\[
\mathbf{T}_{\mathbf{Smooth}}^{W_{1}}\circ\mathbf{T}_{\mathbf{Smooth}}^{W_{2}%
}=\mathbf{T}_{\mathbf{Smooth}}^{W_{1}\otimes_{\mathbf{R}}W_{2}}%
\]
for any objects $W_{1}$ and $W_{2}$ in the category $\mathbf{Weil}%
_{\mathbf{R}}$.

\item  We have
\[
\mathbf{T}_{\mathbf{Smooth}}^{W}\left(  M^{N}\right)  =\mathbf{T}%
_{\mathbf{Smooth}}^{W}\left(  M\right)  ^{\mathbf{T}_{\mathbf{Smooth}}%
^{W}\left(  N\right)  }%
\]
for any objects $M$\ and $N$\ in the category $\mathcal{K}_{\mathbf{Smooth}}$.
\end{itemize}

\item  Given a morphism $\varphi:W_{1}\rightarrow W_{2}$ in the category
$\mathbf{Weil}_{\mathbf{R}}$,
\[
\alpha_{\varphi}:\mathbf{T}_{\mathbf{Smooth}}^{W_{1}}\Rightarrow
\mathbf{T}_{\mathbf{Smooth}}^{W_{2}}%
\]
is a natural transformation subject to the following conditions:

\begin{itemize}
\item  We have
\[
\alpha_{\mathrm{id}_{W}}^{\mathbf{Smooth}}=\mathrm{id}_{\mathbf{T}^{W}}%
\]
for any identity morphism $\mathrm{id}_{W}:W\rightarrow W$ in the category
$\mathbf{Weil}_{\mathbf{R}}$.

\item  We have
\[
\alpha_{\psi}^{\mathbf{Smooth}}\circ\alpha_{\varphi}^{\mathbf{Smooth}}%
=\alpha_{\psi\circ\varphi}^{\mathbf{Smooth}}%
\]
for any morphisms $\varphi:W_{1}\rightarrow W_{2}$ and $\psi:W_{2}\rightarrow
W_{3}$ in the category $\mathbf{Weil}_{\mathbf{R}}$.

\item  Given objects $M$\ and $N$\ in the category $\mathcal{K}%
_{\mathbf{Smooth}}$, the diagram
\[%
\begin{array}
[c]{ccc}%
\begin{array}
[c]{c}%
\mathbf{T}_{\mathbf{Smooth}}^{W_{1}}\left(  M\right)  ^{\mathbf{T}%
_{\mathbf{Smooth}}^{W_{1}}\left(  N\right)  }\\
\parallel\\
\mathbf{T}_{\mathbf{Smooth}}^{W_{1}}\left(  M^{N}\right)
\end{array}
& \underrightarrow{\alpha_{\varphi}^{\mathbf{Smooth}}\left(  M\right)
^{\mathbf{T}_{\mathbf{Smooth}}^{W_{1}}\left(  N\right)  }} & \mathbf{T}%
_{\mathbf{Smooth}}^{W_{2}}\left(  M\right)  ^{\mathbf{T}_{\mathbf{Smooth}%
}^{W_{1}}\left(  N\right)  }\\%
\begin{array}
[c]{ccc}%
\alpha_{\varphi}^{\mathbf{Smooth}}\left(  M^{N}\right)  & \downarrow &
\qquad\qquad
\end{array}
& \nearrow\mathbf{T}_{\mathbf{Smooth}}^{W_{2}}\left(  M\right)  ^{\alpha
_{\varphi}^{\mathbf{Smooth}}\left(  N\right)  } & \\%
\begin{array}
[c]{c}%
\mathbf{T}_{\mathbf{Smooth}}^{W_{2}}\left(  M^{N}\right) \\
\parallel\\
\mathbf{T}_{\mathbf{Smooth}}^{W_{2}}\left(  M\right)  ^{\mathbf{T}%
_{\mathbf{Smooth}}^{W_{2}}\left(  N\right)  }%
\end{array}
&  &
\end{array}
\]
is commutative.

\item  Given an object $W$\ and a morphism $\varphi:W_{1}\rightarrow W_{2}%
$\ in the category $\mathbf{Weil}_{\mathbf{R}}$, the diagrams
\[%
\begin{array}
[c]{ccc}%
\mathbf{T}_{\mathbf{Smooth}}^{W}\circ\mathbf{T}_{\mathbf{Smooth}}^{W_{1}} &
\begin{array}
[c]{c}%
\mathbf{T}_{\mathbf{Smooth}}^{W}\circ\alpha_{\varphi}^{\mathbf{Smooth}}\\
\Rightarrow
\end{array}
& \mathbf{T}_{\mathbf{Smooth}}^{W}\circ\mathbf{T}_{\mathbf{Smooth}}^{W_{2}}\\
\parallel &  & \parallel\\
\mathbf{T}_{\mathbf{Smooth}}^{W\otimes_{\mathbf{R}}W_{1}} &
\begin{array}
[c]{c}%
\Rightarrow\\
\alpha_{\mathrm{id}_{W}\otimes_{\mathbf{R}}\varphi}^{\mathbf{Smooth}}%
\end{array}
& \mathbf{T}_{\mathbf{Smooth}}^{W\otimes_{\mathbf{R}}W_{2}}%
\end{array}
\]
and
\[%
\begin{array}
[c]{ccc}%
\mathbf{T}_{\mathbf{Smooth}}^{W_{1}\otimes_{\mathbf{R}}W} &
\begin{array}
[c]{c}%
\alpha_{\varphi\otimes_{\mathbf{R}}\mathrm{id}_{W}}^{\mathbf{Smooth}}\\
\Rightarrow
\end{array}
& \mathbf{T}_{\mathbf{Smooth}}^{W_{2}\otimes_{\mathbf{R}}W}\\
\parallel &  & \parallel\\
\mathbf{T}_{\mathbf{Smooth}}^{W_{1}}\circ\mathbf{T}_{\mathbf{Smooth}}^{W} &
\begin{array}
[c]{c}%
\Rightarrow\\
\alpha_{\varphi}^{\mathbf{Smooth}}\circ\mathbf{T}_{\mathbf{Smooth}}^{W}%
\end{array}
& \mathbf{T}_{\mathbf{Smooth}}^{W_{2}}\circ\mathbf{T}_{\mathbf{Smooth}}^{W}%
\end{array}
\]
are commutative.
\end{itemize}

\item  Given an object $W$ in the category $\mathbf{Weil}_{\mathbf{R}}$, we
have
\[
\mathbf{T}_{\mathbf{Smooth}}^{W}\left(  \mathbb{R}_{\mathbf{Smooth}}\right)
=\mathbb{R}_{\mathbf{Smooth}}\otimes_{\mathbf{R}}W
\]

\item  Given a morphism $\varphi:W_{1}\rightarrow W_{2}$ in the category
$\mathbf{Weil}_{\mathbf{R}}$, we have
\[
\alpha_{\varphi}^{\mathbf{Smooth}}\left(  \mathbb{R}_{\mathbf{Smooth}}\right)
=\mathbb{R}_{\mathbf{Smooth}}\otimes_{\mathbf{R}}\varphi
\]
\end{enumerate}
\end{proposition}

\section{From the Old Kingdom to the New One}

\begin{notation}
We write
\[
i_{\mathbf{Smooth}}:\mathbf{Smooth}\rightarrow\mathcal{K}_{\mathbf{Smooth}}%
\]
for the functor
\begin{align*}
i_{\mathbf{Smooth}}\left(  \underline{M}\right)   &  :W\in\mathrm{Obj\,}%
\mathbf{We}i\mathbf{l}_{\mathbf{R}}\mapsto\underline{\mathbf{T}}%
_{\mathbf{Smooth}}^{W}\underline{M}\in\mathrm{Obj\,}\mathcal{K}%
_{\mathbf{Smooth}}\\
i_{\mathbf{Smooth}}\left(  \underline{M}\right)   &  :\varphi\in
\mathrm{Mor\,}\mathbf{We}i\mathbf{l}_{\mathbf{R}}\mapsto\underline{\alpha
}_{\varphi}^{\mathbf{Smooth}}\left(  \underline{M}\right)  \in\mathrm{Mor}%
\mathcal{K}_{\mathbf{Smooth}}%
\end{align*}
provided with an object object $\underline{M}$\ in the category
$\mathbf{Smooth}$, and
\[
i_{\mathbf{Smooth}}\left(  f\right)  \left(  W\right)  =\underline{\mathbf{T}%
}_{\mathbf{Smooth}}^{W}f:\underline{\mathbf{T}}_{\mathbf{Smooth}}%
^{W}\underline{M}_{1}\rightarrow\underline{\mathbf{T}}_{\mathbf{Smooth}}%
^{W}\underline{M}_{2}%
\]
provided with a morphism $f:\underline{M}_{1}\rightarrow\underline{M}_{2}$\ in
the category $\mathbf{Smooth}$ and an object $W$ in the category
$\mathbf{We}i\mathbf{l}_{\mathbf{R}}$. The restriction of $i_{\mathbf{Smooth}%
}$ to the subcategory $\mathbf{Mf}$\ is denoted by
\[
i_{\mathbf{Mf}}:\mathbf{Mf}\rightarrow\mathcal{K}_{\mathbf{Smooth}}%
\]
\end{notation}

\begin{theorem}
\label{t5.1}Given an object $W$ in the category $\mathbf{We}i\mathbf{l}%
_{\mathbf{R}}$, the diagram
\[%
\begin{array}
[c]{ccc}%
\mathbf{Mf} & \underrightarrow{i_{\mathbf{Mf}}} & \mathcal{K}_{\mathbf{Smooth}%
}\\
\underline{\mathbf{T}}_{\mathbf{Mf}}^{W}\downarrow &  & \downarrow
\mathbf{T}_{\mathbf{Smooth}}^{W}\\
\mathbf{Mf} & \overrightarrow{i_{\mathbf{Mf}}} & \mathcal{K}_{\mathbf{Smooth}}%
\end{array}
\]
is commutative.
\end{theorem}

\begin{proof}
Given an object $\underline{M}$\ in the category $\mathbf{Mf}$, we have
\begin{align*}
&  \left(  \mathbf{T}_{\mathbf{Smooth}}^{W}\circ i_{\mathbf{Mf}}\right)
\left(  \underline{M}\right) \\
&  =i_{\mathbf{Mf}}\left(  \underline{M}\right)  \circ\left(  W\otimes
_{\mathbf{R}}\cdot\right) \\
&  =\underline{\mathbf{T}}_{\mathbf{Mf}}^{W\otimes_{\mathbf{R}}\cdot
}\underline{M}\\
&  =\underline{\mathbf{T}}_{\mathbf{Mf}}^{\cdot}\left(  \underline{\mathbf{T}%
}_{\mathbf{Mf}}^{W}\underline{M}\right) \\
&  =i_{\mathbf{Mf}}\left(  \underline{\mathbf{T}}_{\mathbf{Mf}}^{W}%
\underline{M}\right) \\
&  =\left(  i_{\mathbf{Mf}}\circ\underline{\mathbf{T}}_{\mathbf{Mf}}%
^{W}\right)  \left(  \underline{M}\right)
\end{align*}
Given a morphism
\[
\underline{f}:\underline{M}_{1}\rightarrow\underline{M}_{2}%
\]
in the category $\mathbf{Mf}$, we have
\begin{align*}
&  \left(  \mathbf{T}_{\mathbf{Smooth}}^{W}\circ i_{\mathbf{Mf}}\right)
\left(  \underline{f}\right) \\
&  =i_{\mathbf{Mf}}\left(  \underline{f}\right)  \circ\left(  W\otimes
_{\mathbf{R}}\cdot\right) \\
&  =\underline{\mathbf{T}}_{\mathbf{Mf}}^{W\otimes_{\mathbf{R}}\cdot
}\underline{f}\\
&  =\underline{\mathbf{T}}_{\mathbf{Mf}}^{\cdot}\left(  \underline{\mathbf{T}%
}_{\mathbf{Mf}}^{W}\underline{f}\right) \\
&  =i_{\mathbf{Mf}}\left(  \underline{\mathbf{T}}_{\mathbf{Mf}}^{W}%
\underline{f}\right) \\
&  =\left(  i_{\mathbf{Mf}}\circ\underline{\mathbf{T}}_{\mathbf{Mf}}%
^{W}\right)  \left(  \underline{f}\right)
\end{align*}
\end{proof}

\begin{theorem}
\label{t5.2}Given a morphism $\varphi:W_{1}\rightarrow W_{2}$ in the category
$\mathbf{We}i\mathbf{l}_{\mathbf{R}}$, the diagram
\[%
\begin{array}
[c]{ccc}%
i_{\mathbf{Mf}}\circ\underline{\mathbf{T}}_{\mathbf{Mf}}^{W_{1}} &
\begin{array}
[c]{c}%
i_{\mathbf{Mf}}\circ\underline{\alpha}_{\varphi}^{\mathbf{Mf}}\\
\Rightarrow
\end{array}
& i_{\mathbf{Mf}}\circ\underline{\mathbf{T}}_{\mathbf{Mf}}^{W_{2}}\\
\parallel &  & \parallel\\
\mathbf{T}_{\mathbf{Smooth}}^{W_{1}}\circ i_{\mathbf{Mf}} &
\begin{array}
[c]{c}%
\Rightarrow\\
\alpha_{\varphi}^{\mathbf{Smooth}}\circ i_{\mathbf{Mf}}%
\end{array}
& \mathbf{T}_{\mathbf{Smooth}}^{W_{2}}\circ i_{\mathbf{Mf}}%
\end{array}
\]
is commutative.
\end{theorem}

\begin{proof}
Given an object\ $\underline{M}$ in the category $\mathbf{Mf}$, we have
\begin{align*}
&  \left(  i_{\mathbf{Mf}}\circ\underline{\alpha}_{\varphi}^{\mathbf{Mf}%
}\right)  \left(  \underline{M}\right) \\
&  =i_{\mathbf{Mf}}\left(  \underline{\alpha}_{\varphi}^{\mathbf{Mf}}\left(
\underline{M}\right)  \right) \\
&  =\underline{\mathbf{T}}_{\mathbf{Mf}}^{\cdot}\left(  \underline{\alpha
}_{\varphi}^{\mathbf{Mf}}\left(  \underline{M}\right)  \right) \\
&  =\underline{\alpha}_{\varphi}^{\mathbf{Mf}}\left(  \underline{\mathbf{T}%
}_{\mathbf{Mf}}^{\cdot}\left(  \underline{M}\right)  \right) \\
&  =\alpha_{\varphi}^{\mathbf{Smooth}}\left(  i_{\mathbf{Mf}}\left(
\underline{M}\right)  \right) \\
&  =\left(  \alpha_{\varphi}^{\mathbf{Smooth}}\circ i_{\mathbf{Mf}}\right)
\left(  \underline{M}\right)
\end{align*}
\end{proof}

\section{\label{s7}Microlinearity}

\begin{definition}
\label{d7.1}Given a category $\mathcal{K}$\ endowed with a functor
$\mathbf{T}^{W}:\mathcal{K}\rightarrow\mathcal{K}$\ for each object $W$ in the
category $\mathbf{We}i\mathbf{l}_{\mathbf{R}}$ and a natural transformation
$\alpha_{\varphi}:\mathbf{T}^{W_{1}}\Rightarrow\mathbf{T}^{W_{2}}$\ for each
morphism $\varphi:W_{1}\rightarrow W_{2}$\ in the category $\mathbf{We}%
i\mathbf{l}_{\mathbf{R}}$, an object $M$\ in the category $\mathcal{K}$\ is
called \underline{microlinear} if any limit diagram $\mathcal{D}$\ in the
category $\mathbf{We}i\mathbf{l}_{\mathbf{R}}$\ makes the diagram
$\mathbf{T}^{\mathcal{D}}M$\ a limit diagram in the category $\mathcal{K}$,
where the diagram $\mathbf{T}^{\mathcal{D}}M$\ consists of objects
\[
\mathbf{T}^{W}M
\]
for any object $W$\ in the diagram $\mathcal{D}$\ and morphisms
\[
\alpha_{\varphi}\left(  M\right)  :\mathbf{T}^{W_{1}}M\rightarrow
\mathbf{T}^{W_{2}}M
\]
for any morphism $\varphi:W_{1}\rightarrow W_{2}$\ in the diagram
$\mathcal{D}$.
\end{definition}

\begin{proposition}
\label{t7.1}Every manifold as an object in the category $\mathbf{Smooth}$\ is microlinear.
\end{proposition}

\begin{proof}
This can be established in three steps.

\begin{enumerate}
\item  The first step is to show that $\mathbf{R}^{n}$ is micorlinear for any
natural number $n$, which follows easily from
\[
\underline{\mathbf{T}}_{\mathbf{Mf}}^{W}\mathbf{R}^{n}=\underline{\mathbf{T}%
}_{\mathbf{Smooth}}^{W}\mathbf{R}^{n}=W^{n}%
\]
and
\[
\underline{\alpha}_{\varphi}^{\mathbf{Mf}}\left(  \mathbf{R}^{n}\right)
=\underline{\alpha}_{\varphi}^{\mathbf{Smooth}}\left(  \mathbf{R}^{n}\right)
=\varphi^{n}%
\]
for any morphism $\varphi:W_{1}\rightarrow W_{2}$\ in the category
$\mathbf{Weil}_{\mathbf{R}}$.

\item  The second step is to show that any open subset of $\mathbf{R}^{n}$\ is
microlinear in homage to the result in the first step.

\item  The third step is to establish the desired result by remarking that a
smooth manifold is no other than an overlapping family of open subsets of
$\mathbf{R}^{n}$.
\end{enumerate}

The details can safely be left to the reader.
\end{proof}

\begin{theorem}
\label{t7.2}The embedding
\[
i_{\mathbf{Smooth}}:\mathbf{Smooth}\rightarrow\mathcal{K}_{\mathbf{Smooth}}%
\]
maps smooth manifolds to microlinear objects in the category $\mathcal{K}%
_{\mathbf{Smooth}}$.
\end{theorem}

\begin{proof}
Let $\mathcal{D}$\ be a limit diagram in the category $\mathbf{Weil}%
_{\mathbf{R}}$. Let $\underline{M}$\ be a smooth manifold in the category
$\mathbf{Smooth}$. Given an object $W$\ in the category $\mathbf{Weil}%
_{\mathbf{R}}$, the diagram $\left(  \mathbf{T}_{\mathbf{Smooth}}%
^{\mathcal{D}}\left(  i_{\mathbf{Smooth}}\left(  \underline{M}\right)
\right)  \right)  \left(  W\right)  $, which consists of objects
\[
\underline{\mathbf{T}}_{\mathbf{Mf}}^{W^{\prime}\otimes_{\mathbf{R}}%
W}\underline{M}=\underline{\mathbf{T}}_{\mathbf{Mf}}^{W\otimes_{\mathbf{R}%
}W^{\prime}}\underline{M}=\underline{\mathbf{T}}_{\mathbf{Mf}}^{W^{\prime}%
}\left(  \underline{\mathbf{T}}_{\mathbf{Mf}}^{W}\underline{M}\right)
\]
for any object $W^{\prime}$\ in the category $\mathbf{Weil}_{\mathbf{R}}$\ and
morphisms
\[%
\begin{array}
[c]{ccc}%
\underline{\mathbf{T}}_{\mathbf{Mf}}^{W_{1}\otimes_{\mathbf{R}}W}\underline{M}%
& \underrightarrow{\underline{\alpha}_{\varphi\otimes_{\mathbf{R}}%
\mathrm{id}_{W}}^{\mathbf{Mf}}\left(  \underline{M}\right)  } & \underline
{\mathbf{T}}_{\mathbf{Mf}}^{W_{2}\otimes_{\mathbf{R}}W}\underline{M}\\
\parallel &  & \parallel\\
\underline{\mathbf{T}}_{\mathbf{Mf}}^{W\otimes_{\mathbf{R}}W_{1}}\underline{M}%
&  & \underline{\mathbf{T}}_{\mathbf{Mf}}^{W\otimes_{\mathbf{R}}W_{2}%
}\underline{M}\\
\parallel &  & \parallel\\
\underline{\mathbf{T}}_{\mathbf{Mf}}^{W_{1}}\left(  \underline{\mathbf{T}%
}_{\mathbf{Mf}}^{W}\underline{M}\right)  & \overrightarrow{\underline{\alpha
}_{\varphi}^{\mathbf{Mf}}\left(  \underline{\mathbf{T}}_{\mathbf{Mf}}%
^{W}\underline{M}\right)  } & \underline{\mathbf{T}}_{\mathbf{Mf}}^{W_{2}%
}\left(  \underline{\mathbf{T}}_{\mathbf{Mf}}^{W}\underline{M}\right)
\end{array}
\]
for any morphism $\varphi:W_{1}\rightarrow W_{2}$\ in the category
$\mathbf{Weil}_{\mathbf{R}}$, is a limit diagram in the category
$\mathbf{Smooth}$, because $\underline{\mathbf{T}}_{\mathbf{Mf}}^{W}%
\underline{M}$\ is a microlinear object in the category $\mathbf{Smooth}$\ in
homage to Proposition \ref{t7.1}. Therefore the diagram $\mathbf{T}%
_{\mathbf{Smooth}}^{\mathcal{D}}\left(  i_{\mathbf{Smooth}}\left(
\underline{M}\right)  \right)  $\ is a limit diagram in the category
$\mathcal{K}_{\mathbf{Smooth}}$ thanks to Theorem 7.5.2 and Remarks 7.5.3 in
\cite{sch}.
\end{proof}

\section{Transversal Limits}

\begin{definition}
A cone $\mathcal{D}$\ in the category $\mathbf{Smooth}$\ is called a
\underline{transversal limit diagram} if the diagram $\mathbf{T}%
_{\mathbf{Smooth}}^{W}\mathcal{D}$\ is a limit diagram for any object $W$\ in
the category $\mathbf{Weil}_{\mathbf{R}}$. In this case, the vertex of the
cone is called a \underline{transversal limit}.
\end{definition}

It is easy to see that

\begin{proposition}
\label{t6.1}A transversal limit diagram is a limit diagram, so that a
transversal limit is a limit.
\end{proposition}

\begin{proof}
Since
\[
\mathbf{T}_{\mathbf{Smooth}}^{\mathbf{R}}\mathcal{D}=\mathcal{D}%
\]
for any cone $\mathcal{D}$\ in the category $\mathbf{Smooth}$, the desired
conclusion follows immediately.
\end{proof}

What makes the notion of a transversal limit significant is the following theorem.

\begin{theorem}
\label{t6.2}The embedding
\[
i_{\mathbf{Smooth}}:\mathbf{Smooth}\rightarrow\mathcal{K}_{\mathbf{Smooth}}%
\]
maps transversal limit diagrams in the category $\mathbf{Smooth}$\ to limit
diagrams in the category $\mathcal{K}_{\mathbf{Smooth}}$.
\end{theorem}

\begin{proof}
This follows directly in homage to Theorem 7.5.2 and Remarks 7.5.3 in
\cite{sch}.
\end{proof}

Now we are going to show that the above embedding preserves vertical Weil
functors, as far as fibered manifolds are concerned. Let us recall the
definition of vertical Weil functor given in \cite{nishi3}.

\begin{definition}
Let us suppose that we are given a left exact category $\mathcal{K}$\ endowed
with a functor $\mathbf{T}^{W}:\mathcal{K}\rightarrow\mathcal{K}$\ for each
object $W$ in the category $\mathbf{We}i\mathbf{l}_{\mathbf{R}}$ and a natural
transformation $\alpha_{\varphi}:\mathbf{T}^{W_{1}}\Rightarrow\mathbf{T}%
^{W_{2}}$\ for each morphism $\varphi:W_{1}\rightarrow W_{2}$\ in the category
$\mathbf{We}i\mathbf{l}_{\mathbf{R}}$. Given a morphism $\pi:E\rightarrow
M$\ in the category $\mathcal{K}$, its \underline{vertical Weil functor}
$\overrightarrow{\mathbf{T}}^{W}\left(  \pi\right)  $\ is defined to be the
equalizer of the parallel morphisms
\[
\mathbf{T}^{W}\left(  E\right)
\begin{array}
[c]{c}%
\underrightarrow{\qquad\qquad\qquad\qquad\mathbf{T}^{W}\left(  \pi\right)
\qquad\qquad\qquad\qquad}\\
\overrightarrow{\mathbf{T}^{W}\left(  \pi\right)  }\,\mathbf{T}^{W}\left(
M\right)  \,\overrightarrow{\alpha_{W\rightarrow\mathbf{R}}\left(  M\right)
}\,\mathbf{T}^{\mathbf{R}}\left(  M\right)  \,\overrightarrow{\alpha
_{\mathbf{R}\rightarrow W}\left(  M\right)  }%
\end{array}
\mathbf{T}^{W}\left(  M\right)
\]
\end{definition}

\begin{lemma}
\label{t6.3}The equalizer of the above diagram in the category
$\mathbf{Smooth}$\ is transversal, as far as $\pi:E\rightarrow M$ is a fibered
manifold in the sense of 2.4 in \cite{kol}.
\end{lemma}

\begin{proof}
The proof is similar to that in Proposition \ref{t7.1}.

\begin{enumerate}
\item  In case that $E=\mathbf{R}^{m+n}$, $M=\mathbf{R}^{m}$, and $\pi$ is the
canonical projection, the equalizer is the canonical injection
\[
\mathbf{R}^{m}\times W^{n}\rightarrow W^{m+n}=\mathbf{T}^{W}\left(  E\right)
\]
and it is easy to see that it is transversal.

\item  Then we prove the statement in case that $E=U\times V$, $M=U$, and
$\pi$ is the canonical projection, where $U$ is an open subset of
$\mathbf{R}^{m}$, and $V$ is an open subset of $\mathbf{R}^{n}$.

\item  The desired statement in full generality follows from the above case by
remarking that the fiber bundle $\pi:E\rightarrow M$\ is no other than an
overlapping family of such special cases.
\end{enumerate}

The details can safely be left to the reader.
\end{proof}

\begin{theorem}
\label{t6.4}Given an object $W$ in the category $\mathbf{We}i\mathbf{l}%
_{\mathbf{R}}$ and a fibered manifold $\pi:E\rightarrow M$\ in the category
$\mathbf{Smooth}$, we have
\[
i_{\mathbf{Smooth}}\left(  \overrightarrow{\underline{\mathbf{T}}%
}_{\mathbf{Smooth}}^{W}\left(  \pi\right)  \right)  =\overrightarrow
{\mathbf{T}}_{\mathbf{Smooth}}^{W}\left(  i_{\mathbf{Smooth}}\left(
\pi\right)  \right)
\]
\end{theorem}

\begin{proof}
In homage to Theorems \ref{t5.1} and \ref{t5.2}, the functor
$i_{\mathbf{Smooth}}$\ maps the diagram
\[
\underline{\mathbf{T}}_{\mathbf{Smooth}}^{W}\left(  E\right)
\begin{array}
[c]{c}%
\underrightarrow{\qquad\qquad\qquad\qquad\qquad\qquad\underline{\mathbf{T}%
}_{\mathbf{Smooth}}^{W}\left(  \pi\right)  \qquad\qquad\qquad\qquad\qquad}\\
\overrightarrow{\underline{\mathbf{T}}_{\mathbf{Smooth}}^{W}\left(
\pi\right)  }\,\underline{\mathbf{T}}_{\mathbf{Smooth}}^{W}\left(  M\right)
\,\overrightarrow{\underline{\alpha}_{W\rightarrow\mathbf{R}}^{\mathbf{Smooth}%
}\left(  M\right)  }\,\underline{\mathbf{T}}_{\mathbf{Smooth}}^{\mathbf{R}%
}\left(  M\right)  \,\overrightarrow{\underline{\alpha}_{\mathbf{R}\rightarrow
W}^{\mathbf{Smooth}}\left(  M\right)  }%
\end{array}
\underline{\mathbf{T}}_{\mathbf{Smooth}}^{W}\left(  M\right)
\]
in the category $\mathbf{Smooth}$\ into the diagram
\[
\mathbf{T}_{\mathbf{Smooth}}^{W}\left(  i_{\mathbf{Smooth}}\left(  E\right)
\right)
\begin{array}
[c]{c}%
\underrightarrow{\qquad\qquad\qquad\mathbf{T}_{\mathbf{Smooth}}^{W}\left(
i_{\mathbf{Smooth}}\left(  \pi\right)  \right)  \qquad\qquad\qquad}\\%
\begin{array}
[c]{c}%
\overrightarrow{\mathbf{T}_{\mathbf{Smooth}}^{W}\left(  i_{\mathbf{Smooth}%
}\left(  \pi\right)  \right)  }\,\mathbf{T}_{\mathbf{Smooth}}^{W}\left(
i_{\mathbf{Smooth}}\left(  M\right)  \right)  \,\\
\overrightarrow{\alpha_{W\rightarrow\mathbf{R}}^{\mathbf{Smooth}}\left(
i_{\mathbf{Smooth}}\left(  M\right)  \right)  }\,\mathbf{T}_{\mathbf{Smooth}%
}^{\mathbf{R}}\left(  i_{\mathbf{Smooth}}\left(  M\right)  \right)  \,\\
\overrightarrow{\alpha_{\mathbf{R}\rightarrow W}^{\mathbf{Smooth}}\left(
i_{\mathbf{Smooth}}\left(  M\right)  \right)  }%
\end{array}
\end{array}
\mathbf{T}_{\mathbf{Smooth}}^{W}\left(  i_{\mathbf{Smooth}}\left(  M\right)
\right)
\]
in the category $\mathcal{K}_{\mathbf{Smooth}}$. Since the equalizer of the
former diagram is transversal by Lemma \ref{t6.3}, it is preserved by the
functor $i_{\mathbf{Smooth}}$\ by Theorem \ref{t6.2}, so that the desired
result follows.
\end{proof}

\begin{corollary}
Given a morphism $\varphi:W_{1}\rightarrow W_{2}$ in the category
$\mathbf{Weil}_{\mathbf{R}}$ and a fibered manifold $\pi:E\rightarrow M$\ in
the category $\mathbf{Smooth}$, we have
\begin{align*}
& i_{\mathbf{Smooth}}\left(  \underline{\overrightarrow{\alpha}}_{\varphi
}^{\mathbf{Smooth}}\left(  \pi\right)  \right) \\
& =\alpha_{\varphi}^{\mathbf{Smooth}}\left(  \left(  i_{\mathbf{Smooth}%
}\left(  \pi\right)  \right)  \right)
\end{align*}
\end{corollary}

\end{document}